\documentclass[12pt]{amsart}

\title{Centralisers of Formal Maps}
\author{Anthony G. O'Farrell}
\address{Mathematics and Statistics\\Maynooth University\\Co Kildare\\W23 HW31\\Ireland}
\email{anthony.ofarrell@mu.ie}

\RequirePackage{amsfonts, amssymb, amsmath, datetime, url}
\date{\today:\currenttime}


\RequirePackage{mathtools}

\DeclarePairedDelimiter{\floor}{\lfloor}{\rfloor}

\newcommand{\dist}{\textup{dist}}

\newcommand{\F}{\mathcal{F}}
\newcommand{\G}{\mathcal{G}}
\newcommand{\gl}{\textup{gl}}
\newcommand{\GL}{\textup{GL}}
\newcommand{\hcf}{\textup{hcf}}
\newcommand{\HOT}{\textup{HOT}}

\newcommand{\M}{\mathcal{M}}
\newcommand{\N}{\mathbb{N}}
\newcommand{\Q}{\mathbb{Q}}

\newcommand{\spt}{\textup{spt}}
\newcommand{\TL}{\tilde L}
\newcommand{\wt}{\textup{wt}}
\newcommand{\Z}{\mathbb{Z}}

\newcommand{\ONE}{1\hskip-4pt1}
\newcommand{\pw}[1]{^{\circ #1}}
\newcommand{\half}{{\textstyle\frac12}}

\newcommand{\ignore}[1]{}

\newtheorem{theorem}{Theorem}

\newtheorem{corollary}{Corollary}[section]
\newtheorem{proposition}{Proposition}[section]

\theoremstyle{definition}
\newtheorem{definition}{Definition}
\newtheorem{conjecture}{Conjecture}
\newtheorem*{remark*}{Remark}

\begin{document}

\begin{abstract}
We consider the formal maps in any finite dimension $d$ with coefficients 
in an integral domain $K$ with identity. Those invertible under 
formal composition form a group $\mathcal{G}$. 
We consider the centraliser $C_g$ of an element $g\in\mathcal{G}$ which 
has infinite order and	is tangent to the identity of $\mathcal{G}$. If $g$ has infinite order and $K$ is a field of characteristic zero we show that $C_g$ contains an isomorphic copy of the additive group $(K,+)$. If $g$ has infinite order and $K$ has positive characteristic we show that $C_g$ contains an uncountable abelian subgroup. 
The proofs are quite different in finite characteristic and in characteristic zero, but are connected by so-called sum functions.
\end{abstract}

\maketitle

\section{Introduction}
We begin by fixing some terminology and notation. 

\subsection{Monic monomials}\label{SS:S}
By a \emph{monic monomial in $d$ variables} we mean an element
$x^i:= x_1^{i_1}\cdots x_d^{i_d}$, where $x=(x_1,\ldots x_d)$
and $i\in\Z_+^d$ is a multi-index with nonnegative entries.
The \emph{degree} of the monic monomial $x^i$ is $|i|:= i_1+\cdots+i_d$.
Let $S=S_d$ denote the set of all monic monomials in $d$ variables.
So $S$
has elements $1$,$x_1$,$\ldots$,$x_d$,$x_1^2$,$x_1x_2$,$\ldots$,
$x_1^3$,$x_1x_2^2$, and so on.
If $m=x^i$, then we denote $x_j^{i_j}$ by $m_j$. For instance,
$(x_1x_2^2x_3^3)_2=x_2^2$.

The set $S$ has the structure
of a commutative semigroup with identity, where the product
is defined by $x^i\cdot x^j:= x^{i+j}$.  The semigroup
$S$ has cancellation, and if $m$ and $n$ belong to $S$,
then we call $m$ \emph{a factor} of  $p=m\cdot n$, we write $m|p$,
and write $p/m=n$. Each nonempty subset $A\subset S$ has a highest
common factor, which we denote by $\hcf(A)$.

\subsection{Formal power series}
Let $K$ be an integral domain with identity, and $d\in\N$.
We form the $K$-algebra $\mathcal{F}:=\mathcal{F}_d := K[[x_1,\ldots,x_d]]$
of all formal power series in $d$ variables, with coefficients in
$K$.
An element $f\in\mathcal{F}$ is a formal sum
$$ f = \sum_{m\in S} f_m m, $$
where $f_m\in K$ for each $m\in S$.
Addition is done term-by-term, as is scalar multiplication,
and multiplication by summing all the coefficients
of terms from the factors corresponding to monomials
with the same product. More precisely,
$$ (f\cdot g)_m = \sum_{p\in S} \sum_{q\in S, pq=m} f_pg_q, \forall m\in S.$$
Equivalently,
$$ (f\cdot g)_m = \sum_{p\in S, p|m} f_p g_{m/p}. $$

The formal series
$1$, which has $1_1=1$ and $1_m=0$
for all other monic monomials $m$, is the multiplicative identity
of $\F$.

In the remainder of this section, all summations will be over
indices drawn from $S$, so the formula for the product becomes just
$$ (f\cdot g)_m = \sum_{p|m} f_p g_{m/p}. $$

The map $f\mapsto f_1$ is a surjective $K$-algebra homomomorphism
from $\F$ onto $K$.

For $f=\sum_m f_mm \in\F$, we set
$$  \spt(f): = \{m\in S: f_m\not=0\}. $$
If $f\not=0$, then $\spt(f)$ is nonempty, and 
we define the \emph{vertex of $f$} to be the monic monomial
$$ v(f):= \hcf( \spt(f) ). $$
It is easy to see that
$$ v(f)v(g) | v(fg), \ \forall f,g\in\F.$$
Equality holds in dimension $d=1$, i.e. 
the index of the lowest-order nonzero
term in $f(x)g(x)$ is the sum of the
indices of the lowest-order terms in
$f(x)$ and $g(x)$.

\begin{proposition}\label{P:v(f)-factor}
If $f\in\F$ is nonzero, then there exists $h\in\F$
with $v(h)=1$ and $f=v(f)h$.
\end{proposition}
\begin{proof}
$v(f)$ is a factor of each $m\in\spt(f)$,
so we can define
$$ h:= \sum_{m, f_m\not=0} f_m\cdot (m/v(f)) $$
and we have
an $h\in\F$ and $f=v(f)h$.
The fact that $v(f)=\hcf(\spt(f)$ implies that
for each $j\in\{1,\ldots,d\}$ there exists
$p\in\spt(f)$ with $f_p\not=0$ and $p_j=v(f)_j$.
Thus $q:=p/v(f)$ belongs to $\spt(h)$ and has $q_j=1$
(i.e. $q$ `does not involve $x_j$'). Thus
$v(h)=1$.
\end{proof}

If $f=ph$ with $p\in S$ and $h\in\F$,
then we write $h=f/p$. Note that $h$
is uniquely determined by $f$ and $p$,
because $h_m = f_{m/p}$ whenever
$h_m\not=0$.

\begin{proposition}
$\F_d$ is an integral domain with identity
	for each $d\in\N$.
\end{proposition}
\begin{proof}
We use induction on $d$.

When $d=1$, the result follows from
the fact that $v(fg)=v(f)v(g)$

Suppose $\F_d$ is an integral domain, and 
$F_{d+1}$ is not.  
	
	Choose two nonzero $f,f'\in F_{d+1}$,
with $ff'=0$. 	
	Replacing $f$ by $f/v(f)$
and $f'$ by $f'/v(f')$, we may assume
that $v(f)=1=v(f')$.  We may regard
monomials in $d$ variables as monomials
in $d+1$ variables that do not involve
$x_{d+1}$, and define
$$ h:= \sum_{m\in S_d} f_mm,\quad
h':= \sum_{m\in S_d} f'_mm.$$
Then
$$ hh' = \sum_{m\in S_d} (ff')_mm,$$
the sum of all the terms in $ff'$ that do not
involve $x_{d+1}$. Thus $hh'=0$, so by
hypothesis $h=0$ or $h'=0$.
	But $h=0$ means $x_{d+1}|v(f)$,
	contradicting $v(f)=1$.  Similarly,
	$h'=0$ is impossible. 
\end{proof}

For any ring $R$ with identity, we denote by $R^\times$
the multiplicative group
of invertible elements, or units, of $R$.

We denote by $\mathcal{M}$ the ideal
$$ \mathcal{F}x_1+\cdots+\mathcal{F}x_d 
=\{f\in\mathcal{F}: f_1=0\}. $$
(This ideal is maximal if and only if $K$ is a field.
In that case, $\mathcal{F}$ is the disjoint union
of $\mathcal{F}^\times$ and $\mathcal{M}$. If $K$
is not a field, then $\mathcal{F}$ is a local ring
if and only if $K$ is local.)

\begin{proposition}\label{P:F^times}
$$ \F^\times =\{ f\in\F: f_1\in K^\times\},$$
so the map $f\mapsto f_1$ is a surjective group homomorphism
from $\F^\times\to K^\times$.
\end{proposition}
\begin{proof}
Let $f\in\F$ have $f_1\in K^\times$. Take
$\alpha:= (f_1)^{-1}\in K$.  Then
$\alpha f = 1 + h$, with $h\in\M$, and we may define
$$ k:= 1 - h + h^2 - h^3 +\cdots \in\F, $$
where the sum makes sense because $v(h^r)$ has order
at least $r$ for each $r\in\N$. Then
$k\alpha f=1$, so $f\in\F^\times$.
This proves that
$$ \{ f\in\F: f_1\in K^\times\}\subset \F^\times.$$

The opposite inclusion is clear,
because if $f\in\F^\times$, and $h=f^{-1}$,
then
$$ 1 = (fh)_1 = f_1h_1, $$
so $f_1\in K^\times$.
\end{proof}

\subsection{Formal maps}
By $\M^d$ we denote (as usual) the Cartesian product
$\M\times\cdots\times \M$ of $d$ factors $\M$,
so an element $f\in\M^d$ is a $d$-tuple $(f_1,\ldots,f_d)$,
with each $f_j\in\M$.

The formal composition $f\circ g$ is defined for $f\in\F$ and
$g\in\M^d$, as follows. First, the composition $m\circ g$ of a
monomial $m=x^i$ with $g$ is $g_1^{i_1}\cdots g_d^{i_d}$,
where the products and powers use the multiplication
of the ring $\F$. Then
$$ f\circ g:= \sum_m f_m\cdot (m\circ g). $$
The sum makes sense because for a given monomial $p\in S$,
the coefficient of $p$ in $m\circ g$ is zero except
for a finite number of $m\in S$; in fact it is zero
once the degree of $m$ exceeds the degree of $p$.
Thus the value
$$ (f\circ g)_p = \sum_{m} f_m\cdot (m\circ g)_p $$
is a finite sum in the ring $K$, and makes sense.

We think of elements of $\M^d$ as \emph{formal self-maps
of $K^d$ fixing $0$}.
The formal composition $f\circ g$ is defined for
$f\in\M^d$ and $g\in\M^d$ by
$$ f\circ g := (f_1\circ g,\ldots,f_d\circ g). $$
With this operation, $\M^d$ becomes a semigroup
with identity; the identity is the element
$$\ONE:=(x_1,x_2,\ldots,x_d).$$
We denote the group of invertible elements
of this semigroup by $\G$.

For $g\in\M^d$, we define \emph{the linear part
of $g$} to be the element of $\gl(d,K)$ with
$(i,j)$ entry given by
$$ L(g)_{ij}:= (g_i)_{x_j}, $$
i.e. the coefficient of the first-degree monomial $x_j$
in the $i$-th component $g_i$ of $g$.

The map $L: \M^d\to\gl(d,K)$ is a semigroup homomorphism,
and its restriction $L:= L|\G$ to the invertible maps
is a group homomorphism $L:\G\to\GL(d,K)$.

We say that an element $g\in\M^d$ is \emph{tangent
to the identity} if $L(g)=\ONE$.

\begin{proposition}
Let $g\in\M^d$. Then $g\in G$ if and only if
$L(g)\in\GL(d,K)$.
\end{proposition}
\begin{proof}
If $g$ is invertible in $\M^d$, then its inverse
$h$ has $h\circ g=\ONE$, and this implies
that the matrix product $L(g)L(h)$ is the identity
matrix. Thus $L(g)\in\GL(d,K)$.

For the converse, suppose $L(g)$ is an invertible
matrix, with inverse $H$. We can also regard $L(g)$
as an element of $\G$, by setting
$$ L(g)_i = \sum_{j=1}^d L(g)_{ij} x_j. $$
If we regard $H$ in the same way,
then $H$ is the compositional inverse of $L(g)$, and we can write
$g= L(g)\circ H\circ g$, so it suffices to
show that $H\circ g$ is invertible in $\M^d$.
Now $H\circ g$ is tangent to the identity,
so we just have to
show that all $g'\in\M^d$ of the form
$$ g' = \ONE + h, $$
where $L(h)=0$, are invertible.  But it is
straightforward to check that such $g'$
are inverted by
$$ h':= \ONE - h +h\circ h - h\circ h\circ h + h\circ h\circ h\circ h
+\cdots.$$
\end{proof}

We remark that a matrix $T\in\gl(d,K)$ is invertible
in $\gl(d,K)$
if and only if its determinant $\det(T)$ belongs to
$K^\times$.  The condition is necessary because
the map $\det$ sends products in $\gl(d,K)$
to products in $K$, and it is sufficient
because when $\det(T)\in K^\times$ we may use
the usual adjugate-transpose construction to
construct an inverse for $T$.

In order to avoid confusion, we prefer to use the notation
$g\pw{k}$ for the $k$-times repeated composition. Thus
$g\pw{2}= g\circ g$, $g\pw3=g\circ g\circ g$, and so on,
and the formula used in the foregoing proof becomes
$$ (1+h)\pw{-1} = \ONE + \sum_{k=1}^\infty (-1)^k h\pw{k}. $$

\subsection{Main Results}\label{SS:results}
\begin{theorem}\label{T:main-char-zero}
Let $K$ be a field
of characteristic zero,
and suppose $g\in\G$ is tangent to the identity
and not equal to the identity.  Then there
is an injective homomorphism from $(K,+)$
into the centraliser of $g$ in $\G$.
\end{theorem}
\begin{theorem}\label{T:main-char-positive}
Let $K$ be an integral domain with identity
having finite characteristic $c$,
and suppose $g\in\G$ is tangent to the identity
and not of finite order.  Then there
is an homomorphism from $(\Z_c,+)$,
the additive group of the $c$-adic integers,
into the centraliser of $g$ in $\G$, and the
image of this homomorphism is uncountable.
\end{theorem}
Combining these, we have our main conclusion:
\begin{theorem}[Main Result]\label{T:main}
Let $K$ be an integral domain with identity,
which is either of finite characteristic or
is an uncountable field. Let $d\in\N$, and let $\G$ be the group
of formal self-maps of $K^d$ fixing zero. 
Then each element of $\G$
tangent to the identity has uncountable
centraliser in $\G$.
\end{theorem}

These results are not new in dimension $d=1$. See
\cite{OFS, Lubin1, Lubin2}. 

In Theorem \ref{T:main-char-zero} it is possible to 
relax the hypothesis that $K$ be a field. 
It works on the weaker hypothesis that $K$ be an
integral domain of characteristic zero having
a certain technical property (called \textit{$\rho$-intactness,}
described in Subsection \ref{SS:rho-intact}
of Section \ref{S:sum-functions} below). For instance,
it works for $K=\Z_p$, the ring of
$p$-adic integers corresponding to any prime $p$,
and for any ring that contains the rational field $\Q$.

\subsection{Wider Context}
We came to this subject because of our interest
in reversibility \cite{OFS}.  We would like to
characterise the reversible elements of well-known
groups.  The power series or formal map groups with real or complex
coefficients are usually considered as a preliminary step
in approaching groups of diffeomorphisms or 
biholomorphic maps.  We would also hope that
progress on the formal map groups will lead to
progress on Riordan-style groups, semidirect
products of a group of invertibles and
a group of automorphims.  The formal maps
we study in the present paper are precisely
the $K$-algebra automorphisms of the
$K$-algebras $\F_d$.  

It is standard procedure, in attempting to
understand reversibility, to start with centralisers
\cite{OFS}.  The results of the present paper
are a first step.  However, it remains to
give a characterisation of the full centraliser
of the maps $g$ we consider.  It is plausible
that in many cases the one-parameter group
we identify will be the full centraliser,
or will have finite index in the centraliser.
We expect this to be the generic situation.
However, if the map $g$ is the direct product
$(h,k)$ of two similar maps in lower dimensions, then
the centraliser of $g$ will have an 
abelian two-parameter subgroup.  We
conjecture that whenever the one-parameter
group we construct here does not have
finite index in the whole centraliser,
the map $g$ is conjugate to a product,
but this conjecture is unproven to date.
We hope that further progess in this direction
might allow us to prove another conjecture,
namely that a reversible formal map can always
be reversed by a map of finite order.

\section{Sum functions}\label{S:sum-functions}
We make no claim to originality for the content
of this section and the next.  These are included for
expository purposes and to set up notation and terminology
for later use.  The mathematical content must
be largely familiar to people who have thought
about these matters.
\subsection{}
Let $K$ be an integral domain with identity. 
We denote the field of fractions of $K$ by $\hat K$, and regard
$K$ as a subring of $\hat K$. 

Let $\pi_K$ denote the group homomorphism from
$(\Z,+)$ into $(K,+)$ such that $\pi_K(1_\Z) = 1_K$.
An induction argument shows that 
it is also a ring homomorphism.
We denote the image $\pi_K(\Z)$ by $\Z_K$. The ring
$\Z_K$ is isomorphic to the quotient ring
$\Z/(c)$, where $c$ is the characteristic of $K$.
If $c>0$, then $\Z_K$ is the \emph{prime field of $K$}.
In characteristic zero, we have $\Q_K:= \hat\Z_K\subset\hat K$
(and $\pi_K$ extends to a field isomorphism from
$\Q$ onto $\Q_K$),
but it may happen that there are 
elements in $\Z_K$ that are noninvertible in $K$.

The \emph{basic sum-functions} with respect to $K$
are the $\rho_m=\rho_{m,K}:\Z_+\to K$, defined inductively for $m\in\Z_+$
by:
$$ \begin{array}{rcl}
\rho_0(k) &:=& 1_K, \forall k\in\Z_+,\\
\rho_{m+1}(k) &:=& \sum_{r=0}^k \rho_m(r), \forall k\in\Z_+,
\end{array}
$$

\begin{definition}
A function $f:\Z_+\to K$ is a \emph{sum-function over $K$}
if it belongs to the $K$-linear span of the basic sum 
functions, i.e. there exist $m\in\Z_+$ and
$\lambda_0$,$\ldots$,$\lambda_m\in K$ such that
$$ f(k) = \sum_{i=0}^m \lambda_i\cdot\rho_i(k), \forall k\in\Z_+.$$
\end{definition}

We denote the set of all sum-functions by $\Sigma_K$,
or just $\Sigma$, when the context is clear.

\subsection{The case $K=\Z$}
It is easy to see that
\begin{equation}\label{E:1}
  \rho_{m,\Z}(k) = \binom{m+k}{m} 
= \binom{m+k}{k}, 
\end{equation}
whenever $m,k\in\Z_+$. 
Thus $\rho_{m,\Z}(k) = p(k)$,
where $p(t)\in\Q[t]$ is a polynomial over $\Q$, having
degree $m$ and leading coefficient $\frac1{m!}$.
Thus each sum function $f\in\Sigma_{\Z}$
coincides on $\Z_+$ with some polynomial $p(t)$ over
$\Q$, and maps $\Z_+$ into $\Z$.  Since $p(t)$
agrees with $f(t)$ on the infinite set $\Z_+$,
if follows that $p$ is uniquely determined by $f$.
We abuse the notation and denote $p(t)$ by
$f(t)$.  For instance, 
$$ \rho_{m,\Z}(t) = \frac{t(t-1)\cdots(t-m+1)}{m!}. $$

According to P\'olya's  definition \cite{Wiki1915}
a polynomial $p(t)\in\Q[t]$ is 
\emph{integer-valued} if 
$p(n)\in\Z$ whenever $n\in\Z$.

\begin{proposition}\label{P:integer-valued-is-sum-function}
The sum-functions over $\Z$ are the same as
 the restrictions to $\Z_+$ of 
the integer-valued polynomials.
\end{proposition}
\begin{proof} 
Given a sum function $f\in\Sigma_{\Z}$,
let $f(t)\in\Q[t]$ be the corresponding polynomial.
If $f(t)$ has degree $n$, then it is uniquely
determined by the values $f(0)$,$f(1)$,$\ldots$,$f(n)$,
and it can be evaluated at each point $x\in\Q$ by
Newton's interpolation formula:
$$\begin{array}{rcl}
 f(x) &=& f(0) + f[0,1]x +f[0,1,2]x(x-1)+
\cdots
\\
&&\qquad+f[0,1,\ldots,n]x(x-1)\cdots(x-n),
\end{array} $$
where $f[0,1,\ldots,m]$ denotes the usual
divided difference. Observing that  
$f[0,1,\ldots,m]$ takes the form of
some integer divided by $m!$, and recalling that
$m!$ divides any product of $m$ consecutive
positive integers, we see that $f(x)\in\Z$
whenever $x$ is a negative integer. Thus
$f(t)$ is an integer-valued polynomial.

\smallskip
For the converse, suppose $p(t)$ is an
integer-valued polynomial.  We wish to see that
$p|\Z_+$ is a sum-function.

If $\deg(p)=0$, then $p=p(0)$
is an integral multiple of $\rho_0$. 
Proceeding inductively, suppose
$m\in\N$ and we are given that each polynomial over $\Q$ of
degree less than $m$ that maps $\Z$ into $\Z$
gives a $\Z$-linear 
combination of 
$\rho_0$,$\ldots$,$\rho_{m-1}$. 
Fix $p(t)\in\Q[t]$ of degree $m$,
and suppose it maps $\Z$ into $\Z$. Then
$q(t):= p(t)-p(t-1)$ belongs to $\Q[t]$,
has degree $m-1$, and maps $\Z$ into $\Z$, so there
exist $\lambda_j\in \Z$ such that
$$ p(k) - p(k-1) 
= \sum_{j=0}^{m-1} \lambda_j\cdot\rho_j(k),$$
whenever $k\in\Z_+$. Thus for $k\in\Z_+$ we have
$$ p(k) = \sum_{r=1}^k \sum_{j=0}^{m-1}\lambda_j\cdot\rho_j(k) + p(0), $$
$$ = \sum_{j=1}^m \lambda_{j-1}\cdot\rho_j(k) + p(0)\rho_0(k). $$
So $p$ is a sum-function, a $\Z$-linear
combination of 
$\rho_0$,$\ldots$,$\rho_m$.
\end{proof}

In particular, for each polynomial $p(t)\in\Z[t]$,
the restriction $p|\Z_+$ is a sum-function over
$\Z$.  

\begin{corollary}\label{C:k^n-is-sum-function}
	For each $n\in\Z_+$, the function
	$k\mapsto k^n$ belongs to $\Sigma_{\Z}$.
\qed
\end{corollary}

From the proof of Proposition
\ref{P:integer-valued-is-sum-function} we conclude:
\begin{corollary}
The basic sum-functions over $\Z$
form a basis for the $\Z$-module (free abelian group)
of all integer-valued polynomials.  Moreover, 
for each $n\in\Z_+$, the subspace of integer-valued polynomials of
degree at most $n$ is the span of the first
$n+1$ basic sum-functions.
\qed\end{corollary}

Using equation \eqref{E:1}, one sees that 
the coefficient of $\rho_n$ in the expression
of $t^n$ as an integral combination of
basic sum-functions is $\frac1{n!}$.

We define the \emph{degree} of a sum-function over $\Z$
to be the degree of the corresponding integer-valued polynomial.

\begin{corollary}\label{C:Sigma_Z-ring}
$\Sigma_{\Z}$ forms a ring under pointwise operations.
\end{corollary}
\begin{proof}
The pointwise product of two integer-valued
polynomials is obviously an integer-valued polynomial.
Thus $\Sigma_{\Z}$ is closed under pointwise
products, as well as pointwise sums. 
\end{proof}

We remark that the proof is easily modified to show
that a given polynomial of degree $n$
(over any field of characteristic
zero containing $\Z$) maps $\Z\to\Z$ as soon as
it maps any $n+1$ consecutive integers into $\Z$. 

\subsection{General $K$}
Now consider an arbitrary  integral domain $K$ with identity.

\begin{proposition}\label{P:Sigma_K-description}
	(1) $\Sigma_K $
is the $K$-linear span of the maps
$$ k\mapsto \pi_K\left(\binom{m+k}{m} \right). $$
	\\(2)
	$\Sigma_K $
is also the $K$-linear span $K\cdot (\pi_K\circ \Sigma_{\Z})$
of the set
$$ \{ \pi_K\circ (p|\Z_+): p(t)
\textup{ is an integer-valued polynomial}\}.
$$\\
	(3) 
	$\Sigma_K $
	contains the set $K[t]\circ\pi_K$ of all the maps
$$ k\mapsto p(\pi_K(k)), \ (p(t)\in K[t]). $$
\end{proposition}
\begin{proof}
	Obviously, 
$\rho_{m,K} = \pi_K\circ\rho_{m,\Z}$.
so part (1) follows at once from equation \eqref{E:1}.

	(2) Follows from Proposition 
\ref{P:integer-valued-is-sum-function}.
  
	(3) Follows from Corollary 
\ref{C:k^n-is-sum-function}.
\end{proof}

\begin{corollary}\label{C:Sigma_K-is-algebra}
The set $\Sigma_K$ of sum-functions is a $K$-algebra of
functions from $\Z_+\to K$, when  
equipped with pointwise operations
of addition, multiplication and scalar
multiplication.
\end{corollary}
\begin{proof}
This follows from Corollary \ref{C:Sigma_Z-ring}
and part (2) of Proposition \ref{P:Sigma_K-description}.
\end{proof}

From part (1) of Proposition \ref{P:Sigma_K-description}
we deduce:
\begin{proposition}\label{C:coefficient}
The pointwise product of $\rho_m$ and $\rho_n$ is
a $\Z_K$-linear combination of $\rho_j$'s, where $j$
ranges from $0$ to $m+n$,
and the coefficient of $\rho_{m+n}$
is $\pi_K(\binom{m+n}{n})$. 
\qed
\end{proposition}

We note that the latter coefficient may be zero,
depending on the characteristic of $K$.

\subsection{Characteristic zero}\label{SS:rho-intact}
If $K$ has characteristic $0$, then the basic sum-functions
are linearly-independent over $\hat K$ (when considered
as functions from $\Z_+$ into $\hat K$), and each
element $f$ of $\Sigma_K$ takes the form
$p\circ\pi_K$, where $p(t)\in\hat K[t]$ is a polynomial
over $\hat K$, and indeed $p(t)$ has its coefficients
in the product ring $K\cdot\Q_K$. The polynomial
$p$ is uniquely-determined by $f$, and we denote it by
$\hat f$.

Indeed, the polynomial $\hat\rho_{m,K}$ has degree $m$,
and a nontrivial $\hat K$-linear relationship between the $\rho_{m,K}$
would entail a $\hat K$-linear relationship between
the $\hat\rho_{m.K}$, and that cannot occur between
polynomials having distinct degrees.

Hence each $f\in\Sigma_K$ has a \emph{unique}
expression as a $K$-linear combination of basic
sum-functions, and we may define the \emph{degree}
of $f$ to be the least $m\ge0$ such that the coefficient
of $\rho_{n,K}$ is zero for all $n>m$.  This is
then the same as the degree of the polynomial $\hat f$.

\begin{proposition}\label{P:K-K-is-sum-function}
Let $K$ have characteristic zero, and let $p(t)\in\hat K[t]$.
If $p(a)\in K$ whenever $a\in K$, then
$p\circ\pi_K$ is a sum-function over $K$.
\end{proposition}
\begin{proof}
Suppose $p$ maps $K$ into $K$.
If $\deg(p)=0$, then $p\circ\pi=p(0)\rho_0$
is a sum function of degree $0$. 
Proceeding inductively, suppose
$m\in\N$ and we are given that each polynomial over $\hat K$ of
degree less than $m$ that maps $K$ into $K$
gives a sum-function. Fix $p(t)\in\hat K[t]$ of degree $m$,
and suppose it maps $K$ into $K$. Then
$q(t):= p(t)-p(t-1)$ belongs to $\hat K(t)$,
has degree $m-1$, and maps $K$ into $K$, so there
exist $\lambda_j\in K$ such that
$$ (p\circ\pi)(k) - (p\circ\pi)(k-1) 
= \sum_{j=0}^{m-1} \lambda_j\cdot\rho_j(k),$$
whenever $k\in\Z_+$. Thus for $k\in\Z_+$ we have
$$ (p\circ\pi)(k) = \sum_{r=1}^k \sum_{j=0}^{m-1}\lambda_j\cdot\rho_j(k) + p(0), $$
$$ = \sum_{j=1}^m \lambda_{j-1}\cdot\rho_j(k) + p(0)\rho_0(k). $$
So $p\circ\pi$ is a sum-function.
\end{proof}

A simple-minded converse to Proposition 
\ref{P:K-K-is-sum-function} would say that
if $h$ is a sum-function, then $\hat h$ maps
$K$ into $K$. But this is not
true, in general. For instance, taking
$K:=\Z[y]$ for an indeterminate $y$, the
polynomial corresponding to the basic sum-function
$\rho_2$ is $\hat\rho_2(t)=\half t(t+1)$, and yet
$\hat\rho_2(y) = \half y(y+1)$ does not belong to
$K$. It does not help to assume $K$, or even
$(K,+)$, finitely-generated,
because the same $\hat\rho_2$ does not
map the ring of Gaussian integers into itself.

Prompted by this, let us say that an integral
domain $K$ is 
\emph{$\rho$-intact} if 
$K$ has characteristic zero and
$$ \hat\rho_m(a) \in K,\ \forall a\in K\ \forall m\in\Z_+.$$

Examples  of $\rho$-intact domains are $\Z$,
the ring $\Z_p$ of $p$-adic integers
corresponding to a prime $p$, and all fields
of characteristic zero.
It is also easy to see that if
$K$ is any integral domain of characteristic zero,
then $K\Q_K$ and $\hat K$ are $\rho$-intact. 
(In the case of $K\Q_K$, you could use the
fact that the product of each $m$-term
arithmetic progression of integers is
divisible by $m!$.)
define \emph{the $\rho$-intact-envelope of $K$}
to be
$$ K_\rho:= \bigcap\{ L: K\subset L,
L\textup{ is a $\rho$-intact subring of }\hat K\}.
$$
Then $K_\rho$ is $\rho$-intact, and
$K\subset K_\rho\subset K\Q_K$.
Of course, if $K$ is $\rho$-intact, then
$K_\rho=K$.

With this terminology, we have, trivially:
\begin{proposition}
If $K$ is a domain of characteristic zero, then
$\hat h(t)\in(K\Q_K)[t]$ maps $K_\rho$ into $K_\rho$
for each sum-function $h\in\Sigma_K$.
\qed
\end{proposition}

\begin{remark*}
	This account of sum-functions in characteristic zero $K$
	could be recast in the language of 
	Elliott \cite{Elliott}. The $\rho$-intact integral
	domains of characteristic zero are torsion-free examples of
 \textit{binomial rings}, a concept useful in group
	theory that goes back to
	Philip Hall \cite{Hall} (or \cite{Hall+}), and has a substantial literature.
The $\rho$-intact envelope of a $K$ of characteristic zero
	is the same as the binomial
	ring
	$\textup{Bin}^U(K)$ 
	of theorem 7 in \cite{Elliott}, where the
	functor $\textup{Bin}^U$ from the category of 
	commutative rings
	with identity to the category of binomial rings 
	is a left adjoint of the forgetful functor
	from binomial rings to rings.
\end{remark*}

\subsection{Positive characteristic}
In positive characteristic $c$ some sum-functions are not given by 
polynomials
at all.  
All polynomial functions $P\circ \pi_K$, with $P(t)\in K[t]$
have period $c$, because $\pi_K(k+c) = \pi_K(k)$ for
$k\in\Z_+$. But
in characteristic $2$, for instance, 
the sequences $(\rho_m(k))_k$ 
for $m=0$,$1$,$2$,$3$,$4$ begin:
\begin{center}
$1,1,1,1,1,1,1,1,\ldots,$
\\$1,0,1,0,1,0,1,0,\ldots,$
\\$1,1,0,0,1,1,0,0,\ldots,$
\\$1,0,0,0,1,0,0,0,\ldots,$
\\$1,1,1,1,0,0,0,0,\ldots,$
\end{center}
and have periods $1,2,4,4$ and $8$, respectively.

In fact, all sum-functions are periodic in finite-characteristic,
and more:
\begin{theorem}\label{T:sum-fn}
Let the integral domain $K$ have characteristic $c>0$.
Then for each $r\in\Z_+$, 
the set $\{\rho_m: 0\le m< c^r\}$ is a basis for the
$K$-module of all functions $f:\Z_+\to K$ that have 
period $c^r$.
\end{theorem}
To be clear, when we say that $f$ has period $n$
we mean that $f(k+n)=f(k)$ for all $k\in\Z_+$. We do
not mean that $n$ is the least positive period of $f$.

We will prove this theorem in Section \ref{S:T-proof}.

The theorem has several immediate corollaries:
\begin{corollary}
Suppose $K$ has characteristic $c>0$.
\\(1)
The sum-functions over $K$ are the same
as the functions whose least positive period is
some power of $c$.
\\(2) 
The basic sum-functions form a linearly-independent set
over $\hat K$.
\\(3) 
Each sum-function $h\in\Sigma_K$ 
has a unique expression
$$ h(k) = \sum_m^\infty h_m \rho_m(k), \ \forall k\in\Z_+,$$
in which each $h_m\in K$ and only a finite number of
$h_m$ are nonzero.
\qed
\end{corollary}

This allows us to define the
degree of a sum-function as follows:

\begin{definition}
The \emph{degree} of the sum-function
$ h = \sum_m^\infty h_m \rho_m$ is the largest
$m$ having $h_m\not=0$.
\end{definition}

This is consistent with the definitions previously given
in the case of rings $K$ of characteristic zero, although
it no longer relates to the degree of any associated polynomial.

From Corollary \ref{C:coefficient} we deduce:
\begin{proposition}
If $h$ and $h'$ are sum-functions, then
$$ \deg( hh' ) \le \deg(h) + \deg(h'). $$
\qed
\end{proposition}

\section{Block-patterns}\label{S:T-proof}
This section may be skipped by
readers who are only interested in
domains $K$ of characteristic zero.
Suppose the integral domain $K$ has characteristic $c>0$.
We may regard the values $\rho_m(k)\in\Z_K$
as forming an infinite matrix, with entries
drawn from $\{0,1,\ldots,c-1\}$, with rows
indexed by $m=0,1,2,\ldots$ and
columns indexed by $k=0,1,2,\ldots$.
We call this \emph{the basic sum-function matrix},
and denote it by $B$. It is an element of the
matrix ring $\gl(\Z_+,\Z_K)$, and $B_{mk}:= \rho_m(k)$.

It is a symmetric matrix. The top-left
$c\times c$ block provides a pattern
that we call \emph{the template}, and
denote by $B_1$. 
It has the
$c^2$ entries given by
$$ \rho_m(k) = \binom{m+k}{k} = 
\frac{(m+k)\cdots(m+1)m}{k!}, $$
where we note that $k!$ is invertible in
$\Z_K$ when $0\le k<c$. For instance, the templates
in characteristics $3$ and $5$ are
$$
\left(
\begin{matrix}
1&1&1\\1&2&0\\1&0&0
\end{matrix}
\right)
,\textup{ and }
\left(
\begin{matrix}
1&1&1&1&1\\1&2&3&4&0\\1&3&1&0&0\\
1&4&0&0&0\\1&0&0&0&0
\end{matrix}
\right).
$$
For general $r\in\N$, we denote the top left
$c^r\times c^r$ block of $B$ by $B_r$.

The top left $c^2\times c^2$ block $B_2$
falls into $c^2$
blocks, each one $c\times c$, and (as we shall see)
the pattern is that the $(i,j)$-th
$c\times c$ block is obtained from the template
by mutiplying each term (modulo $c$) by
the $(i,j)$-th entry in the template! Thus, denoting
the template by $T$, the patterns in characteristics
$3$ and $5$ are:
$$
\left(
\begin{matrix}
T&T&T\\T&2T&0\\T&0&0
\end{matrix}
\right)
,\textup{ and }
\left(
\begin{matrix}
T&T&T&T&T\\T&2T&3T&4T&0\\T&3T&T&0&0\\
T&4T&0&0&0\\T&0&0&0&0
\end{matrix}
\right).
$$
The block $B_3$, 
is obtained
by iterating this procedure: It has $c^2$
blocks, each $c^2\times c^2$, and they are
obtained from the top left $c^2\times c^2$
block by multiplying it by the appropriate
entry in the template. For instance, you can
discern this structure in the following top left section
of the characteristic $3$ matrix $B$:
$$
\left(
\begin{matrix}
\begin{matrix}
1&1&1\\1&2&0\\1&0&0
\end{matrix}
&
\begin{matrix}
1&1&1\\1&2&0\\1&0&0
\end{matrix}
&
\begin{matrix}
1&1&1\\1&2&0\\1&0&0
\end{matrix}
&
\begin{matrix}
1&1&1\\1&2&0\\1&0&0
\end{matrix}
&
\begin{matrix}
1&1&1\\1&2&0\\1&0&0
\end{matrix}
&
\begin{matrix}
1&1&1\\1&2&0\\1&0&0
\end{matrix}
&
\begin{matrix}
1&1&1\\1&2&0\\1&0&0
\end{matrix}
\\
\begin{matrix}
1&1&1\\1&2&0\\1&0&0
\end{matrix}
&
\begin{matrix}
2&2&2\\2&1&0\\2&0&0
\end{matrix}
&
\begin{matrix}
0&0&0\\0&0&0\\0&0&0
\end{matrix}
&
\begin{matrix}
1&1&1\\1&2&0\\1&0&0
\end{matrix}
&
\begin{matrix}
2&2&2\\2&1&0\\2&0&0
\end{matrix}
&
\begin{matrix}
0&0&0\\0&0&0\\0&0&0
\end{matrix}
&
\begin{matrix}
1&1&1\\1&2&0\\1&0&0
\end{matrix}
\\
\begin{matrix}
1&1&1\\1&2&0\\1&0&0
\end{matrix}
&
\begin{matrix}
0&0&0\\0&0&0\\0&0&0
\end{matrix}
&
\begin{matrix}
0&0&0\\0&0&0\\0&0&0
\end{matrix}
&
\begin{matrix}
1&1&1\\1&2&0\\1&0&0
\end{matrix}
&
\begin{matrix}
0&0&0\\0&0&0\\0&0&0
\end{matrix}
&
\begin{matrix}
0&0&0\\0&0&0\\0&0&0
\end{matrix}
&
\begin{matrix}
1&1&1\\1&2&0\\1&0&0
\end{matrix}
\\
\begin{matrix}
1&1&1\\1&2&0\\1&0&0
\end{matrix}
&
\begin{matrix}
1&1&1\\1&2&0\\1&0&0
\end{matrix}
&
\begin{matrix}
1&1&1\\1&2&0\\1&0&0
\end{matrix}
&
\begin{matrix}
2&2&2\\2&1&0\\2&0&0
\end{matrix}
&
\begin{matrix}
2&2&2\\2&1&0\\2&0&0
\end{matrix}
&
\begin{matrix}
2&2&2\\2&1&0\\2&0&0
\end{matrix}
&
\begin{matrix}
0&0&0\\0&0&0\\0&0&0
\end{matrix}
\\
\begin{matrix}
1&1&1\\1&2&0\\1&0&0
\end{matrix}
&
\begin{matrix}
2&2&2\\2&1&0\\2&0&0
\end{matrix}
&
\begin{matrix}
0&0&0\\0&0&0\\0&0&0
\end{matrix}
&
\begin{matrix}
2&2&2\\2&1&0\\2&0&0
\end{matrix}
&
\begin{matrix}
1&1&1\\1&2&0\\1&0&0
\end{matrix}
&
\begin{matrix}
0&0&0\\0&0&0\\0&0&0
\end{matrix}
&
\begin{matrix}
0&0&0\\0&0&0\\0&0&0
\end{matrix}
\\
\begin{matrix}
1&1&1\\1&2&0\\1&0&0
\end{matrix}
&
\begin{matrix}
0&0&0\\0&0&0\\0&0&0
\end{matrix}
&
\begin{matrix}
0&0&0\\0&0&0\\0&0&0
\end{matrix}
&
\begin{matrix}
2&2&2\\2&1&0\\2&0&0
\end{matrix}
&
\begin{matrix}
0&0&0\\0&0&0\\0&0&0
\end{matrix}
&
\begin{matrix}
0&0&0\\0&0&0\\0&0&0
\end{matrix}
&
\begin{matrix}
0&0&0\\0&0&0\\0&0&0
\end{matrix}
\\
\begin{matrix}
1&1&1\\1&2&0\\1&0&0
\end{matrix}
&
\begin{matrix}
1&1&1\\1&2&0\\1&0&0
\end{matrix}
&
\begin{matrix}
1&1&1\\1&2&0\\1&0&0
\end{matrix}
&
\begin{matrix}
0&0&0\\0&0&0\\0&0&0
\end{matrix}
&
\begin{matrix}
0&0&0\\0&0&0\\0&0&0
\end{matrix}
&
\begin{matrix}
0&0&0\\0&0&0\\0&0&0
\end{matrix}
&
\begin{matrix}
0&0&0\\0&0&0\\0&0&0
\end{matrix}

\end{matrix}
\right)
$$

Each number $m\in\Z_+$ has a unique 
\emph{canonical $c$-adic expansion}
$$ m = \sum_{r=0}^\infty m_r c^r $$
in powers of (the prime) $c$, where $0\le m_r<c$
and only a finite number of $m_r$ are nonzero.
Using the $c$-adic expansions of $m$ and $k$, we
have a formula for $\rho_m(k)$:

\begin{proposition}\label{P:product-formula}
If $m,k\in\Z_+$ have the canonical $c$-adic expansions
$ m = \sum_{r=0}^\infty m_r c^r $ and
$ k = \sum_{r=0}^\infty k_r c^r $, then
\\(1)
\begin{equation}\label{E:formula}
 \rho_m(k) = \prod_{r=0}^\infty \rho_{m_r}(k_r), 
\end{equation}
\\(2) In particular, for each $r\ge1$,
$\rho_{m}(k) =0$ whenever $m<c^r$, $k<c^r$, and
$m+k\ge c^r$.
\end{proposition}
Note that all but a finite number of terms
in the product are equal to $1$.  In fact, since
$m_r=0$ when $r>\log_c m$, we may write
$$ \rho_m(k) = \prod_{r=0}^{\floor{\log_cm}}\prod_{s=0}^{\floor{\log_ck}}
\rho_{m_r}(k_s), $$
where $\floor{\cdot}$ denotes the floor function.
This proposition expresses all entries in the $(\rho_m(k)$
matrix in terms of the upper left $c\times c$ template $B_1$, 
and justifies
the pattern described at the end of the last subsection.

Part (2) may be summarized by saying that the
submatrices $B_r$ are `upper-left triangular',
i.e. have zero entries below the antidiagonal.

\begin{proof} 
It is obvious that formula (\ref{E:formula}) holds
when $m<c$ and $k<c$, i.e. for the entries in the template. 
Both sides of the formula are symmetric in $m$ and $k$,
so it suffices to prove the case $k\le m\ge c$.

We proceed by induction.
Suppose $r\ge1$ and (1) and (2) hold whenever
$m<c^r$ and $k<c^r$. Fix $m$ and $k$
with $c^r\le m<c^{r+1}$,
and $0\le k\le m$. Then $m_r\not=0$,
and $m=m_rc^r+m'$ with $0\le m'<c^r$.
Also $k= k_rc^r+k'$ with $0\le k'<c^r$.

$\rho_m(k)$ is the coefficient of $t^m$ in $(1+t)^{m+k}$
(regarded as a polynomial over $\Z_K$). In standard 
notation, we denote this coefficient by\\ $[t^m](1+t)^{m+k}$.
Thus
$$ \rho_m(k) = [t^m]\left(  (1+t)^{(m_r+k_r)c^r}(1+t)^{m'+k'}
\right).$$
Now $(1+t)^c= 1+t^c$ (since $c=0$ in $K$), 
and repeating this we get
$$ (1+t)^{(m_r+k_r)c^r} = (1+t^{c^r})^{m_r+k_r}. $$
From this factor, only the term in 
$t^{m_rc^r}$ can contribute 
to $\rho_m(k)$ (since $k'<c^r$), so
$$ \rho_m(k) = 
\left(
[t^{m_rc^r}](1+t^{c^r})^{m_r+k_r} 
\right)
\cdot
\left(
[t^b](1+t)^{m'+k'}
\right).
$$
So
$$ \rho_m(k) = \rho_{m_r}(k_r) \cdot \rho_{m'}(k') 
.$$
Applying the assumption (1) to
$m' = \sum_{s<r} m_sc^s$ and $k'=\sum_{s<r} k_sc^s$,
we have 
$$ \rho_{m'}(k') = 
  \prod_{s<r} \rho_{m_s}(k_s), 
$$
so we we get formula (\ref{E:formula}) for the pair
$(m,k)$. 

Turning to (2), we want to show that 
$\rho_m(k)=0$ if $m+k\ge c^{r+1}$.
But if $m+k\ge c^{r+1}$, then $m_r+c_r\ge c$,
and hence $\rho_{m_r}(c_r)=0$
(i.e. the template is upper-left triangular),
and hence the formula gives $\rho_m(k)=0$.

This completes the induction step,
and the proof.
\end{proof}

We now draw some corollaries. First,
the $c^r$-th rows  of the matrix 
take a special form:
\begin{corollary}
For $r\in\Z^+$, we have $\rho_{c^r}(k) = k_r+1$.
\end{corollary}
\begin{proof}
Formula (\ref{E:formula}) gives
$$ \rho_{c^r}(k) = \rho_1(k_r),$$
since $\rho_0(j)=1$ for all $j$.
But $\rho_1(k) = k+1$ (modulo $c$).
\end{proof}

In particular, the first $c^r$ entries in the
$c^r$-th row are all $1$. The $c^r+1$-st entry is
$2$ (equal to zero, in case $c=2$).

By symmetry, we have:
\begin{corollary}
For $r\in\Z^+$, we have $\rho_m(c^r) = m_r+1$.
\qed
\end{corollary}

Combining these, we see that the
upper-left triangular
square submatrix $(B_{mk})$ indexed by
$1\le m,k< c^r$ is `framed' in $B$
by a square of $1$'s, except that that
lower right corner of the frame is a $2$.

By repeatedly applying the Law of Pascal's triangle,
working up from the $c^r$-th row,
we deduce inductively:
\begin{corollary}
For each $r\in\N$
all elements on the antidiagonal of $B_r$
are of the form $\pm1$, and alternate
between $+1$ and $-1$.
\end{corollary}
In other words, for $0\le m,k<c^r$, we have
\\$\rho_m(c^r-m-1)=\pm1$. Note that except
in characteristic $2$, this antidiagonal
has an odd number of entries, and in
characteristic $2$, all the entries are
the same.
\\Thus if we reverse the order of the rows in $B_r$,
we get a lower-triangular matrix having
invertible elements on the diagonal.

\begin{corollary}
$B_r\in\GL(c^r,K)$, so the $K$-linear span
of the rows is $K^{c^r}$, and the
rows are K-linearly-independent. 
\end{corollary}

Next, since $\rho_0(k)=1$ for all $k$, the
formula gives:
\begin{corollary}
$\rho_m$ has period $c^r$ whenever $0\le m<c^r $.
\end{corollary}

\begin{proof}[Proof of Theorem \ref{T:sum-fn}]
The last two corollaries combine to complete the
proof of Theorem \ref{T:sum-fn}. The $\rho_m$
are periodic, with period $c^r$; 
they are already $K$-linearly-independent
as functions on the first $c^r$ nonnegative integers,
and hence \textit{a fortiori} as functions on
$\Z_+$; and their $K$-linear combinations 
give every $K$-valued function on $\Z_+$
of period $c^r$,
because any such function is determined by its values
on the subset $\{0,1,\ldots,c^r-1\}$.  
\end{proof}

\section{Iteration of Formal Maps}
\subsection{Coefficients of iterates}
We now consider the iteration of a map $g\in\G$.  We define
$g^{\circ k}$ for $k\in\Z_+$ by setting $g^{\circ 0}:= \ONE$
and inductively defining $g^{\circ(k+1)} := 
g\circ g^{\circ k}$.  

We also define the backward iterates by
$$g^{\circ{-k}}:= (g^{\circ(-1)})^{\circ k}.$$

For each $g\in\G$, 
the map $k\mapsto g^{\circ k}$ is a homomorphism
from the additive group $(\Z,+)$ into $\G$:
$$  g^{\circ(k+l)} = g^{\circ k}\circ g^{\circ l}, $$
whenever $k,l\in\Z$.

In the same way, we can define the $k$-th iterate $g^{\circ k}$
of an element $g$ belonging to the semigroup $\M^d$,
provided $k\in\Z_+$, but we cannot in general define
it for negative integers $k$.  For each such $g$,
the map $k\mapsto g^{\circ k}$ is a homomorphism
from the additive semigroup $(\Z_+,+)$ into 
the compositional semigroup $(\M^d, \circ)$.

\medskip
An element $g\in\M^d$ has a series expansion
$$ g = \sum_{m\in S} g_m\cdot m, $$
where $g_m = ((x_1\circ g)_m,\ldots,(x_d\circ g)_m)$
belongs to $K^d$. We refer to $g_m$ as \emph{the
$m$-th coefficient of $g$}.  Note that $g_m$ is a
$d$-dimensional vector over $K$.
If we group terms of the same degree,
by defining
$$ L_k(g):=  
\sum_{m\in S, \deg m=k} g_m\cdot m, \ \forall k\in\Z_+,$$
then we get the expansion
$$ g = \sum_{k=1}^\infty L_k(g). $$
Here, $L_1(g)=L(g)$ is the linear part of $g$,
and in general we refer to
to $L_k(g)$ as \emph{the $k$-th homogeneous 
term of $g$}.  This term is a $K^d$-valued 
homogeneous polynomial of degree $k$ with coefficients in $K$,
or, equivalently, it is a $d$-tuple of
homogeneous polynomials of degree $k$ over $K$.

\begin{theorem}\label{T:2} 
Suppose $g\in\G$ is tangent to the identity, 
and let $m\in S$.
Then there is a $d$-tuple of sum-functions 
$P(t)\in \Sigma_K^d$
(depending on $g$ and $m$)
such that $(g\pw k)_m = P(k)$ for each $k\in\N$.
	\end{theorem}

\begin{proof} 
The $j$-th component of the coefficient $(g\pw{k})_m$
is the matrix entry $(x_j\circ g\pw{k})_m$
of $M(1,g)^k$, so we want to show this is
a sum-function of the variable $k$.
Thus it is enough to show that each entry
$(n\circ (g\pw{k}))_m$ is a sum-function in $k$
(with coefficients in $K$).

Consider the hypothesis ($H_s$): \emph{that
the entry 
$(n\circ (g\pw{k}))_m$ is a sum-function in $k$
of degree less than $\deg(m)$ 
whenever $m,n\in S$ and $\deg m<s$.}

This holds for $s=2$, because the only monic monomials
of degree $1$ are the $x_j$ ($j=1,\ldots,d$),
and 
$$ (n\circ (g\pw{k}))_{x_j} = (n\circ\ONE)_{x_j} 
=\left\{
\begin{array}{rl}
1,& n=x_j,\\
0,& \textup{ otherwise}.
\end{array}
\right.
$$

Assume that $s\ge2$ and $H_s$ holds. We claim
that $H_{s+1}$ also holds.  To prove this, we have to show
that for each $m\in S$ of degree $s$, and each
$n\in S$, the entry 
$(n\circ (g\pw{k}))_m$ is a sum-function in $k$
of degree less than $s$.

Fix $m\in S$, of degree $s$. 

First, consider the case $n=x_j$, for some $j\in\{1,\ldots,d\}$,
and let $\alpha_k:= 
(x_j\circ (g\pw{k}))_m$.  
For $k\ge1$, we have
$$ \alpha_{k+1} = (x_j\circ g\circ g\pw{k}) = 
\sum_{p\in S} (x_j\circ g)_p \cdot (p\circ g\pw{k})_m.
$$
If $\deg(p)=0$ or $\deg(p)>s$, then $(p\circ g\pw{k})_m=0$.
If $\deg(p)=1$, then $p=x_i$ for some $i$,
and $(x_j\circ g)_p$ is equal to $1$ or $0$, depending
on whether or not $i=j$.
If $\deg(p)=s$, then $(p\circ g\pw{k})_m = p_m$ and equals
$1$ or $0$, depending on whether or not $p=m$. So
if we let
$$ T:= \{ p\in S: 1<\deg(p)<s \}, $$
and 
$\lambda_p:= (x_j\circ g)_p$, 
then
$$ \alpha_{k+1} = \alpha_k 
+ 
\sum_{p\in T} \lambda_p \cdot (p\circ g\pw{k})_m
+ \lambda_m. $$

Fix $p\in T$.  Since the degree of $p$ is at least $2$,
we can factor $p$ as $x_i\cdot q$ for some
$q\in S$, and then
$$ p\circ g\pw{k} = (x_i\circ g\pw{k})\cdot (q\circ g\pw{k}), $$
so
\begin{equation}\label{E:5}
 (p\circ g\pw{k})_m = 
\sum_{r|m} (x_i\circ g\pw{k})_r\cdot (q\circ g\pw{k})_{m/r}. 
\end{equation}
In this sum, the hypothesis $H_s$ tells us
that the terms are sum-functions in $k$, except perhaps
for the terms $r=1$ and $r=m$. The term with $r=1$
is zero (because $(x_i\circ g\pw{k})_1=0$), and
the term with $r=m$ is also zero, because
$(q\circ g\pw{k})_1=0$.  Thus
$(p\circ g\pw{k})_m$
is a sum-function in $k$.  

It follows that
$$ \alpha_{k+1} =\alpha_k + P(k), $$
where  $P$ is a sum-function in $t$.
Thus
$$ \alpha_k =\sum_{r=1}^{k-1} P(r) + \alpha_1 $$
is a sum-function in $k$, of degree at most $\deg(P)+1$.

It remains to show that the degree of $P$
is less than $\deg(m)$.  By the induction hypothesis,
the $r$-th term in the sum
in Equation (\ref{E:5}) is zero or is the product 
of a sum-function in $k$ of degree less than
$\deg(r)$ and a sum-function of degree less than
$\deg(m/r)$, so it has degree less than or equal to
$$ \deg(r)-1 + \deg(m/r) -1 = \deg(m)-1.$$

Thus $(n\circ g\pw{k})_m$ is a sum-function in $k$
when $n=x_j$. 

Now consider a monomial 
$n\in S$ of degree greater than $1$.  
We can factor $n$ as $x_i\cdot q$ for some $q\in S$
of degree at least $1$, and then
$$ 
(n\circ g\pw{k})_m =  
\sum_{r|m} (x_i\circ g\pw{k})_r\cdot (q\circ g\pw{k})_{m/r}. 
$$
As before, the terms with $r=1$ and $r=m$ have a zero factor,
hence equal zero.  In the nonzero terms, $\deg(r)<s$
and $\deg(m/r)<s$, so by the induction hypothesis
each such term is the product of a sum-function in $k$
of degree less than $\deg(r)$ and another of degree
less than $s/\deg(r)$.  Hence  
$(n\circ g\pw{k})_m$ is a sum-function in $k$ of degree
less than $s$.  

Thus $H_{s+1}$ holds.

By induction, $H_s$ holds for all $s$, and
the theorem is proven.
\end{proof}

From the proof, we note:
\begin{corollary}\label{C:to-T3}
Each component of $P(t)$ is a sum-function
of degree less than the degree of $m$,
and it depends only on the coefficients
$g_p$ of the monomials $p\in S$ of degree less
than or equal to the degree of $m$.
\qed\end{corollary}

In fact, the dependence on the coefficients
of $g$ is polynomial.

\subsection{Orders and degrees}
The order of an element of a group is the least power
of the element that equals the identity, or is
infinity if there is no such power. If we speak about \emph{the order} of 
a formal map $g\in\G$, this is what we mean.
For nonzero power series $f\in\F$, people sometimes use the term
`order' to refer to the least $k\in\Z_+$
such that some monomial of degree $k$ has a
nonzero coefficient, what could
informally be called the `order of vanishing'
of the series.  We refer to this
number as the \emph{lower degree}
of the series.  We take the lower degree of $0$
to be infinity. We extend this to $d$-tuples
$(f_1,\ldots,f_d)\in\F^d$, by defining the lower
degree to be the minimum of the lower degrees
of the components $f_j$.  For $g\in\G$, tangent to $\ONE$, we
refer to the lower degree of $g-\ONE$ as the
\emph{Weierstrass degree of $g$}.

\begin{proposition}
Suppose $K$ has characteristic zero.
If $g\in\G$ is tangent to the identity, and $g\not=\ONE$,
then $g$ has infinite order.
\end{proposition}
\begin{proof}
We may write  
$$ g = \ONE + \sum_{k=r}^\infty L_k, $$
where $L_k=L_k(g)$ and $r\ge2$ is least with $L_r\not=0$.
Then
$$ 
\begin{array}{rcl}
g\pw2(x) &=& \ONE(g(x)) + L_r(g(x)) + \HOT
\\
&=& \ONE + L_r(x) + \HOT + L_r(\ONE) + \HOT
\\
&=& \ONE + 2L_r(x).
\end{array}
$$
Continuing inductively, we get
$$ g\pw{n} =\ONE + nL_r + \HOT. $$
Thus $g\pw{n}\not=\ONE$ for each $n\in\N$.
\end{proof}

In finite characteristic $c$, there are
maps tangent to the identity that
have finite order.  For instance,
in dimension one, take
$$ g(x) = \frac{x}{1+x}. $$
Then one calculates that
$$ g\pw{k}(x) = \frac{x}{1+kx}, $$
so $g\pw{c}(x)=x$, i.e. $g\pw{c}=\ONE$.
However, the argument in the proof of
the proposition shows the following:

\begin{proposition}
Suppose $K$ has characteristic $c>0$.
If $g\in\G$ is tangent to the identity,
then either $g$ has infinite order,
or the order of $g$ is divisible by $c$. 
\qed
\end{proposition}

\subsection{Iterates to nonintegral order}
In this subsection we consider the possibility
of defining iterates $g\pw\lambda$,
where $\lambda$ belongs to the ring $K$. 
This possibility arises mainly in 
characteristic zero.  In the following subsection,
we will consider an alternative procedure
in positive characteristic. 
 
For $g\in\G$, 
Theorem \ref{T:2} tells us that
if $g$ is tangent to the identity, then
for each monomial $m$, there is a sum-function
$P_m(t)$ over $K$ (depending on $g$), 
of degree less than  $\deg m$,
 such that $(g\pw{k})_m = P_m(k)$ for all $k\in\N$.

\begin{theorem}\label{T:3}
Let $g\in\G$ be tangent to the identity, and 
fix a monic monomial $m\in S$ of degree $s$. If
$s$ is not greater than the cardinality
of $\Z_K$, then there exists a unique
polynomial $P_m(t)\in \hat K[t]$ of degree less than $s$
such that
$$  P_m(k) = 
 (g\pw{k})_m \ \forall k\in\Z_+.
$$
\end{theorem}
\begin{proof}
$P_m(t)$ is a polynomial over the field $\hat K$, so it is
uniquely determined by its values at $\deg(P_m)+1$
distinct elements of $K$.  The elements $0_K$,
$1_K$,$\ldots$, $s\cdot1_K$ are distinct,
so the result follows.
\end{proof}

\begin{corollary}\label{C:3}
If $K$ has characteristic zero, then 
for each $m\in S$, the
 polynomial $P_m$ is uniquely determined by $g$.
\end{corollary}

\begin{definition} Suppose $K$ has characteristic zero.
For $g\in\G$, tangent to the identity, and $\alpha\in K$,
we define the \emph{$\alpha$-th iterate $g\pw{\alpha}$ of $g$}
by the formula
$$ g\pw{\alpha}:= \sum_{m\in S} P_m(\alpha)\cdot m, $$
\end{definition}
This might be referred to as the formal-formal iterate,
because it has to do with formal composition, and
the order of iteration is also `formal'.

The series $g\pw{\alpha}$ is a power series over 
the $\rho$-intact envelope $K_\rho$, as opposed to
$K$. We can, if we wish, extend the definition
to allow any $\alpha$ belonging to $K_\rho$,
$K\Q_K$, $\hat K$, or any $\rho$-intact domain that contains
$K$. In case $K$ is the ring of integers, 
the ring of $p$-adic integers or any
field of characteristic zero, the formal-formal
iterate $g\pw{\alpha}$ belongs to the original
group $\G=\G_K$.  

Note that $g\pw{(\pi_K(k))}=g\pw{k}$ for $k\in\Z$.

\begin{theorem}\label{T:4}
Suppose $K$ has characteristic zero. Then
if $g\in\G$ is tangent to the identity, 
we have
$$ g\pw{a}\circ g\pw{b} = g\pw{a+b} $$
whenever $a,b\in K$.
\end{theorem}
\begin{proof}
If a polynomial equation holds at more
points than the degree of the polynomial,
then it holds identically.
The identity
\begin{equation}\label{E:6}
 g\pw{a}\circ g\pw{b} = g\pw{a+b} 
\end{equation}
holds for all $a\in\N$ and $b\in\N$. As a result,
for each $m\in S$ and each index $j\in\{1,\ldots,d\}$,
we have
$$ (x_j\circ (g\pw{a}\circ g\pw{b}))_m = 
(x_j\circ(g\pw{a+b}))_m $$
whenever $a\in\N$ and $b\in\N$.  For fixed $a\in\N$,
this is a polynomial identity in $b$,
so since it holds for all $b\in\N$, it holds
identically for all $b\in K$.  Then, for fixed
$b\in K$ it is a polynomial identity in $a$,
and in the same way it must hold
identically for all $a\in K$.  
Since this holds for all $m\in S$,  all
coefficients agree in the expansion of
the two sides of Equation (\ref{E:6}), and
hence the equation holds for all $a,b\in K$.
\end{proof}

So in characteristic zero, we have
the following fact:
\begin{corollary}\label{C:4-1}
For each $a\in K$, $g\pw{a}$ commutes with $g$,
and $a\mapsto g\pw{a}$ is a group homomorphism
from the abelian group $(K,+)$ into the centraliser of $g$ in $\G_{K_\rho}$.
\qed
\end{corollary}

This allows us to give:
\begin{proof}[Proof of Theorem \ref{T:main-char-zero}]
The homomorphism $a\mapsto g\pw{a}$ is injective,
because $L_k(g\pw{a})=aL_k(g)$ when $k$ is the
Weierstrass degree of $g$.
\end{proof}

Theorem \ref{T:4} has other interesting
consequences:
\begin{corollary}\label{C:4-2}
The coefficients of the compositional inverse $g\pw{-1}$
are given by
$P_m(-1)$.\qed
\end{corollary}

\begin{corollary}
If $g,h\in\G$ are tangent to the identity and commute, then
$$ g\pw{a}\circ h\pw{b} = h\pw{b}\circ g\pw{a}, \ \forall a,b\in K.$$
\qed\end{corollary}

In particular, if $K$ is $\rho$-intact, then the map
$$ (a,b) \mapsto g\pw{a}\circ h\pw{b} $$
gives a group homomorphism from $(K^2,+)$ into
the centraliser $C_\G(g)$ whenever $g$  and $h$ are commuting elements
of $\G$ that are tangent to the identity.  It is obviously
of interest to know when this map is injective, so the following proposition
is worth noting:

\begin{proposition}
Let $g,h\in\G$ be tangent to the identity. If there exists
$p\in K^\times$ such that $h\pw{p}=g$, then $g\pw{1/p}=h$, and
$$ \{h\pw{a} : a\in K\} = \{g\pw{a}: a\in K\} .$$
\end{proposition}

Thus, if we think of $\{h\pw{a}:a\in K\}$ as a `one-parameter subgroup'
(where the `parameter' runs over $K$) of $\G$, then two one-parameter
subgroups in the centraliser of a $g$ 
(tangent to but not equal to $\ONE$)
either coincide or give
a `two-parameter subgroup'.

\begin{conjecture}
We conjecture that if $K$ is a field of characteristic 
zero and $g\in\G$ is
tangent to but not equal to the identity,
then the centraliser 
$C_\G(g)$ 
of $g$ in $\G$ 
is abelian, and is the inner direct product of its torsion
subgroup and a finite number of one-parameter groups of
iterates  $\{h\pw{a}:a\in K\}$.  We expect that, generically
the centraliser will be just $\{g\pw{a}:a\in K\}$, and that
the occurrence of a two-parameter subgroup corresponds to
the possibility of conjugating $g$ to a product map,
	i.e. a formal map of the form
	$(h(x_1,\cdots,x_m),k(x_{m+1},\ldots,x_d))$,
	where $h\in\G_m$ and $k\in\G_{d-m}$
	for some $m\in\N$ with $1\le m<d$.

The conjecture holds in dimension $d=1$ \cite{OFS}.
\end{conjecture}

\subsection{$c$-adic iterates}
Throughout this section, we suppose
that $K$ has positive characteristic $c$.

There does not appear to be any reasonable way to
define iterates $g\pw{\lambda}$ for $\lambda\in K$
in characteristic $c$, but we can extend the scope
of iteration in another way.

We impose the $\M$-adic valuation topology on $\F$, i.e. the topology
induced by any translation-invariant metric that has
$$   \dist(m\cdot f, 0) = 2^{-\deg m} $$
whenever $m\in S$ and $f\in\F^\times$. Thus a sequence
$(f^{(n)})_n$ of power series over $K$ converges to zero
if and only if 
$$ \forall m\in S,\ \exists N\in\N: n>N\implies (f^{(n)})_m=0, $$
i.e. the coefficient of each monomial in the expansion
of $f^{(n)}$ is eventually zero.
We use the $\M$-adic valuation topology on $\F$
to generate a product topology on $\F^d$, and
restrict it to $\M^d$ and $\G$. We refer to all these
as $\M$-adic valuation topologies.

\begin{theorem}\label{T:c-adic-action}
Suppose $g\in G$ has $L(g)=\ONE$. Then there is
a continuous group homomorphism 
$z\mapsto g\pw{z}$  from $(\Z_c,+)$ (the additive
group of the $c$-adic integers) into $\G$ that extends the
iteration map $k\mapsto g\pw{k}$ from $\Z$. 
\end{theorem}
\begin{proof}
By Theorem \ref{T:2}, for each monic monomial $m\in S$,
we have a $d$-tuple of sum-functions $P_m(t)\in\Sigma^d$
over $K$ such that 
$$ g\pw{k} = \sum_m P_m(k) m, \ \forall k\in\Z_+. $$
For each $m$, the component functions of $P_m$ are
 of degree less than
$m$, and hence $P_m$ is periodic of
order $c^r$, where $r$ is the ceiling of $\log_c\deg m$,
i.e. $P_m(k+c^r)=P_m(k)$ for all $k\in\N$.

Each $z\in\Z_c$ has a $c$-adic expansion 
$z=\sum_{s=0}^\infty z_s p^s$, and we now define
$$ P_m(z):= P_m\left( \sum_{s=0}^r z_s p^s\right)\in K. $$ 

Observe that, by the periodicity, 
$$ P_m(z)= P_m\left( \sum_{s=0}^t z_s p^s \right) $$
whenever $t>r$, so that it does not matter
where we truncate the expansion of $z$: as long
as it is far enough out (depending on $m$), 
we get the same value for $P_m(z)$. Also,
if $z\in\Z_0$, then the value of $P_m(z)$
is the same as it was.

Now we define 
$$ g\pw{z}:= \sum_m P_m(z) m. $$
Then $z\mapsto g\pw{z}$ extends the iterates map from
$\Z$ to $\Z_c$.
The extended map is continuous from the $c$-adic valuation
topology on $\Z_c$ to the $\M$-adic valuation topology
on $\G$, because if $c^r$ divides $z-z'$, then
$P_m(z)=P_m(z')$ whenever $\deg m< c^r$.

Since $g\pw{k}\circ g\pw{k'} = g\pw{k'}\circ g\pw{k}$
for all $k,k'\in\Z_0$, and since $\N$ is dense
in $\Z_c$, it follows that
$$ g\pw{z}\circ g\pw{z'} = g\pw{z'}\circ g\pw{z}$$
whenever $z,z'\in\Z_c$, so we have a group
homomorphism, as asserted.

Finally, if $n\in\N$, then the newly-defined
$g\pw{-n}$ --- defined using the $c$-adic expansion of
$-n$ --- is the compositional inverse of $g\pw{n}$
(by the group homomorphism property), and
hence must agree with $(g^{-1})\pw{n}$
(by the uniqueness of inverses in the group $\G$),
so $z\mapsto g\pw{z}$ extends the composition map
from all of $\Z$, and not just from $\Z_+$.
\end{proof}  

\begin{corollary}
If $g\in\G$ is tangent to the identity, then
for each $n\in\N$ not divisible by $c$,
there exists an $n$-th compositional root
$h\in\G$ of $g$, i.e. an element with
$h\pw{n}=g$.
\end{corollary}
\begin{proof}
If $n\in\N$ is not divisible by $c$,
then $n\in\Z_c^\times$, so
$g\pw{1/n}$ is an $n$-th compositional root
of $g$.
\end{proof}

In particular, if $c\not=2$, each $g\in\G$ tangent to the identity
has a compositional square root. In characteristic $2$, we have
roots of every odd order.

We can now conclude:
\begin{proof}[Proof of Theorem \ref{T:main-char-positive}]
If $g\in\G$ is tangent to the identity, 
and has infinite order,
then 
the map $k\mapsto g\pw{k}$ is injective
from $\Z_+$ into $\G$.  Since $g\pw{c^r}\to\ONE$
as $r\uparrow\infty$, we see that 
$\ONE$ is the limit of a sequence of
distinct iterates.

The map
$z\mapsto g\pw{z}$ is 
continuous from the compact space $\Z_c$
so the one-parameter
image subgroup is compact.  Since the identity
is not isolated, there are no isolated points,
so it is a compact metric space without isolated points,
and hence has the cardinality of the continuum.
\end{proof}

We don't see any reason why the map 
$z\mapsto g\pw{z}$ has to be injective
on $\Z_c$, so we can't say that the
one-parameter subgroup is isomorphic
to $(\Z_c,+)$. 

\subsection{Maps having linear part of finite order}
We close with some remarks about
the centralisers of 
formal maps having linear part of finite order.

Let $K$ be any integral domain.
Suppose $h\in\G$ is such that its linear
part $L:=L_1(h)$
is of finite order, say $s$. Then
$g:=h\pw{s}$ is tangent to the identity, and we may apply the
results of the last two subsection to $g$.

\noindent
\\1. If $h$ itself has finite order, and its order is not
a multiple of the characteristic $c$ of $K$
(--- this includes the case $c=0$),
then $h$ is conjugate to 
$L$ (--- compare the argument for lemma 2.1 in \cite{OFZ})
and so $h$ has order $s$. Then 
we can take any $f\in\F$   
with $f_1=1$ and form
$$ \tilde f:= \frac1s\left(
f+f\circ L+\cdots+f\circ L^{\circ(s-1)}
\right).
$$
Then $\tilde f\in\F$, $\tilde f_1=1$ 
and $\tilde f\circ L = \tilde f$.
So
$$ \bar f(x):= \tilde f(x)\cdot\ONE = (\tilde f(x)x_1,\ldots,\tilde f(x)x_d) 
$$
belongs to $\G$, is tangent to the identity, and
commutes with $L$.  It follows that all $\bar f^{\circ n}$
commute with $L$, and then that $\bar f^\lambda$
also commutes with $L$ for all $\lambda\in K$
in characteristic zero or for all $\lambda\in\Z_c$
in positive characteristic $c$. Thus $h$ also has a
large centraliser.

\noindent
\\2. The case when $K$ has positive characteristic $c$ and
$h\in\G$ has a finite order that is a multiple of
$c$ requires further analysis, and we do not
pursue it here.

\noindent
\\3. Observe that if $h$ has infinite order, then so does
$g$, and the fact that
$h$ commutes with each $g\pw{k}$ for $k\in\N$ implies that
$h$ commutes with the whole one-parameter group
(parametrised by $K$ or $\Z_c$, as the case may be) 
determined by $g$.

\noindent
\\4.
If $L(h)\not=\ONE$, then whenever  the characteristic
does not divide $s$, we have a $g\pw{1/s}\not=h$,
so we get a nontrivial element $h':=h\pw{-1}g\pw{1/s}$
of order $s$.  This allows us to factor
$h$ \emph{in its own centraliser}
 as the product of an element of order $s$
and an element tangent to the identity.

\end{document}